 \documentclass{ajour}
\usepackage{epsf}
\usepackage{latexsym}
\usepackage{amsfonts}
\usepackage{amssymb}

\newcommand{\one}{{\bf 1}}

\begin{document}

%


 \authorrunninghead{M.C. de Bondt}
 \titlerunninghead{Jeep variants}





\title{Jeep variants}

\subtitle{Applications of convoy formulations}

\author{Michiel de Bondt}
\affil{Department of mathematics, University of Nijmegen, \\
       Toernooiveld 1, 6525 ED Nijmegen, The Netherlands}

\email{MichieldeB@netscape.net}

%







\abstract{The jeep problem was first solved by O. Helmer
and N.J. Fine.
But not much later, C.G. Phipps formulated a more general
solution. He formulated a so-called convoy or caravan variant
of the jeep problem and reduced the original problem to it.

The convoy idea of Phipps was refined in \cite{opjc}. Here we will
apply this refined idea to several variants of the jeep problem.}


\begin{article}


\section{A FINE JEEP AND DEPOTS TO BE USED} 
\label{Finesec}

Suppose we have a jeep that has a fuel capacity of one tankload of fuel, 
which can ride one distance unit per tankload. Furthermore, the jeep
may set up depots of fuel on any position in the desert for future use.
We call such a jeep from now on a {\em Fine jeep}.

Say we have a Fine jeep that must cross a desert in order to reach an oasis, and 
possibly return to the desert border afterwards. How can this be achieved with a 
minimal amount of fuel, which is available at the desert border? This problem was 
solved first in \cite {helmer} and \cite{fine}.

We now consider the problem of a Fine jeep crossing the desert in order to
reach an oasis, with
both depots to be used and depots to be filled. Although we assume that
there is only one jeep, we formulate the algorithm as a {\em backward convoy
algorithm}, which has been introduced in \cite{opjc} using ideas of \cite{phipps}. 
The algorithm is both for the outward trip case and the return
trip case. For the moment, we assume that the jeep has to end at the oasis 
finally in the outward trip case. 

One of the aspects of a backward convoy algorithm is that compared to a 
normal algorithm, time is in fact eliminated. But in case of depots to
be used, time might matter, since some depot must be reached first, before
its fuel can be used. So in general, there are positions where fuel is
more scarce before than after some depot is reached.

To take this into account, we will split the backward convoy into two 
parts at some depots. One part is the reaching part: a {\em forward refueling
sub\-convoy} with relatively less fuel. The other part is the using part: 
a {\em backward refueling sub\-convoy} with 
relatively more fuel. In case of a round trip to the oasis, one of the 
double jeeps
splits in fact in two single jeeps: one for each sub\-convoy. The single jeep
for the backward refueling convoy is in fact a returning single jeep and 
therefore called a {\em ringle jeep}. The other single jeep is just called
a single jeep.

In a backward convoy, the forward refueling sub\-convoy has at most one 
tankload of fuel, but on the
contrary, the backward refueling sub\-convoy has at most one tankload of 
{\em emptiness} in its tanks. As soon as the backward refueling 
sub\-convoy gets one tankload of emptiness, one of its jeeps is canceled.
This canceling takes only one tankload of fuel, and therefore eliminates all
emptiness in the backward refueling convoy. 

For an easier formulation, we assume by definition that there is always a
forward and a backward refueling sub\-convoy. The backward refueling sub\-convoy
is improper if it is reduced to a convoy without jeeps or an empty ringle 
jeep.

\begin{algo} \label{Finedep}
Start with a single jeep with one tankload of fuel and possibly an empty 
ringle jeep. If there is fuel at the oasis, then do the handler of event 1
first. Ride to the desert border. Each time the forward refueling sub\-convoy
gets out of fuel, a double jeep with one tankload of fuel is added to it. A 
backward refueling sub\-convoy that only consists of one empty ringle jeep
uses fuel of the forward refueling sub\-convoy.

\begin{em}
Event 1: The convoy meets a depot with fuel.
\end{em} \\
{\em Handler:} 
The forward refueling sub\-convoy absorbs as much fuel as possible, but each 
time this sub\-convoy gets more than one tankload of tank fuel, it cancels a 
double jeep with one tankload of fuel. If there is more fuel than the forward 
refueling sub\-convoy can accept, then we call the current position a {\em
saturation point}. At a saturation point, the backward refueling sub\-convoy
absorbs all remaining fuel. It creates a double jeep with one tankload of fuel 
each time it gets more fuel than it can absorb. This new double jeep with one
tankload of fuel can absorb another tankload of fuel, since its tankfuel 
capacity is two tankloads.

If fuel need to be restored at the current position for the party 
organization, or -- even worse -- more fuel has to be put on the current 
position than there was before this handler, then do the handler of event 
2 first. Ride to the desert border with both sub\-convoys. The backward 
refueling sub\-convoy cancels a double jeep with one tankload of fuel each time 
it gets more than one tankload of tank emptiness, since then one double jeep 
less is required to transport all fuel. 

\begin{em}
Event 2: The convoy meets a depot to be filled.
\end{em} \\
{\em Handler:}
Use fuel of the backward refueling sub\-convoy, canceling a double jeep with
one tankload of fuel each time this sub\-convoy gets more than one tankload of tank
emptiness. If more fuel is needed than the backward refueling sub\-convoy can 
give, then the forward refueling sub\-convoy must give the remaining fuel. To do 
this, it creates as many double jeeps with one tankload of fuel each as
necessary.

Eventually at the desert border, the fuel tanks of the forward refueling 
subconvoy are filled up to a level of one tankload. The fuel of the tanks of the
backward convoy above the level of one tankload, as well as all tank fuel from 
the ringle jeep, can be returned at the desert border if that is allowed.
\end{algo}

In case there is a position where fuel can be used but where fuel
has to be put as well, then possibly double jeeps are created that are 
canceled on the same position in algorithm \ref{Finedep}. But an 
algorithm in which this does not happen needs more words. Furthermore,
a position may be a saturation point, even if on balance, it requires
fuel. 

Before we prove the optimality of algorithm \ref{Finedep}, we formulate
a trivial but very important result, which is used implicitly in
\cite{roundtrip} as well.

\begin{proposition}[Split Lemma]
Suppose we have a Fine jeep that rides from $a$ to $b$, using fuel from 
the desert and passing $x$ one or more times. Suppose that this jeep
executes some fuel transportations as well along the way. 

Then the
same consumptions and transportations of fuel can be done by two Fine
jeeps in the following way. One jeep rides from $\min\{a,x\}$ to 
$\min\{b,x\}$ and the other jeep rides from $\max\{a,x\}$ to 
$\max\{b,x\}$, both without passing $x$. If we allow borrowing fuel
at $x$ that is not yet carried to $x$, then both jeeps may ride after 
each other in any order.
\end{proposition}

\begin{proof}
The proof is left as an exercise to the reader.
\end{proof}

Let a {\em rejoining point} be a position where the backward refueling
sub\-convoy gets improper. Now we are ready for our final theorem.

\begin{theorem} \label{Finedepth}
Algorithm \ref{Finedep} is optimal and can be executed as a normal algorithm.
\end{theorem}

\begin{proof}
Let $S$ be a normal algorithm where the jeep reaches the oasis, 
and $T$ be an instance of algorithm \ref{Finedep} that has the same effect 
on fuel depots as $S$, with a ringle jeep if and only if $S$ is a round trip,
but in which more fuel might be used from the desert border. 
We must show that no more fuel is used from
the desert border in $T$. Furthermore, we must show that $T$ can be executed 
as a normal algorithm as well.

We call the following positions {\em special}:
\begin{itemize}

\item Positions with fuel and rejoining points of $T$,

\item The desert border and the oasis.

\end{itemize}      
Let $G = \{G_1, G_2, \ldots, G_r\}$ be the set of special points with 
$G_1 > G_2 > \cdots > G_r = 0$.

Split $S$ into parts $S_i$ on $[G_{i+1}, G_i]$ by way of the split lemma. 
Split $T$ into parts $T_i$ on $[G_{i+1}, G_i]$. First, we prove that
$T$ is optimal. After that, we show that $T$ can be transformed into a normal
algorithm.

In order to prove that $T$ is optimal, we show that each of its parts
$T_i$ is optimal. Assume by induction that $T_1, \ldots, T_{i-1}$ are optimal.
Then in $T_{i-1}$, at least as much fuel is dumped on $G_i$ or at most as 
much fuel is taken from $G_i$ on balance as in $S_{i-1}$. So in $T_i$,
at most as much fuel need to be dumped on $G_i$ or at most as 
much fuel may be taken from $G_i$ on balance as in $S_i$.

In case there is no proper backward convoy on the interval $(G_{i+1},G_i)$ in 
$T$, it follows from \cite[Th.\@ 4.1]{opjc} that we can change $S_i$ into a normal 
algorithm that follows algorithm \ref{Finedep}. So assume that there is a 
proper backward convoy on the interval $(G_{i+1},G_i)$ in $T$. 

There are normal algorithms $T_{i1}$ and $T_{i2}$ corresponding to the 
forward and backward refueling subconvoy of $T$ respectively over the interval 
$[G_{i+1},\allowbreak G_i]$. Since there is no rejoining
point between $G_{i+1}$ and $G_i$ exclusive, no fuel is taken from 
position $G_{i+1}$ in $T_{i2}$.

$S_i$ can be split into a part $S_{i1}$ before reaching $G_i$ for the first 
time and a part $S_{i2}$ after that. If we apply \cite[Th.\@ 4.1]{opjc} on $S_{i1}$, with 
the desert border and the oasis replaced by $G_{i+1}$ and $G_i$ respectively, we see 
that in $T_{i1}$, no more fuel is used as in $S_{i1}$. 

If we apply \cite[Th.\@ 4.1]{opjc} on $S_{i2}$, with the desert border and the oasis 
replaced by $G_i$ and $G_{i+1}$ respectively, we see that we may assume that $S_{i2}$ is
derived from a backward convoy algorithm. In a backward convoy algorithm,
and hence in $S_{i2}$, fuel molecules do not need to cross each other. 
So we may assume that in $S_{i2}$, either $G_i$-fuel does not reach $G_{i+1}$ 
or only $G_i$-fuel is used.

In the former case, $T_{i2}$ is at least as economical as $S_{i2}$, since no 
$G_{i+1}$-fuel is used in $T_{i2}$. In the latter case, it follows from 
\cite[Th.\@ 4.1]{opjc} that the amount of $G_i$-fuel of $T$ and $S$ which
reaches $G_{i+1}$ is at least 
as large in $T_{i2}$ as in $S_{i2}$. It follows that $T_i$ is optimal.

So the optimality of $T$ follows by induction. Next, we need to show that 
$T$ can be executed as a normal algorithm. Notice that each of the parts
$T_{ij}$ can be executed as a normal algorithm. The only problem is to combine
them such that on the special points $G_i$, no temporary underflows of fuel
occur. Assume first that $S$ is a two-way trip. Then $T$ can be executed as a 
normal algorithm in the order $T_{(r-1)1} \ldots T_{11} \allowbreak T_{12} 
\ldots T_{(r-1)2}$. This is because the forward parts $T_{(r-1)1} \ldots 
T_{11}$ are constructed in such a way that only present fuel of the point $G_i$
is used, and in the backward parts $T_{12} \ldots T_{(r-1)2}$, the amount of
fuel first only increases and then only decreases on the points $G_i$. 

In case $S$ is a one-way trip, things are more difficult. Assume first that
in $(G_{i+1},G_i)$, the backward refueling subconvoy is improper. If we do
$T$ in the order $T_{r-1}, \ldots, T_{i+1}, T_i, T_{i-1}, \ldots, T_1$, then 
temporary underflows of fuel might occur, but not during $T_i$. In order to
remove all temporary underflows, we only need to look at blocks of intervals
$(G_{i+1},G_i)$ where the backward convoy is proper. 

Let $(G_j,G_{j-1}),
\ldots, (G_{i+1},G_i)$ be such a block. Split $T_{k2}$ into a part $T_{k21}$ 
before reaching $G_{k+1}$ for the last
time and a part $T_{k22}$ thereafter, and execute $T_{j-1},\ldots,T_i$ in the 
order $T_{(j-1)1} \ldots T_{i1} \allowbreak T_{i21} \ldots T_{(j-1)21} 
\allowbreak T_{(j-1)22} \ldots T_{i22}$. This way, no temporary underflows
occur in $T_{j-1},\ldots,T_i$, as can be shown with essentially the same
arguments as in the round trip case.
\end{proof}

In case the jeep does not need to end at the oasis or the desert border,
the optimal algorithm is the best of the following.
\begin{enumerate}

\item The outward trip variant of \ref{Finedep}.

\item The round trip variant of \ref{Finedep}. 

\item All variants of \ref{Finedep} that start with a ringle jeep, but where
at some rejoining point, the ringle jeep is thrown away. 

\end{enumerate}
The reader may show this.
In the normal algorithm corresponding to 3., the jeep ends where the 
ringle jeep is thrown away.

\section{A CONVOY ALGORITHM FOR DEWDNEY JEEPS}

We already met Dewdney jeeps in Maddex' jeep problem. A Dewdney jeep has a tank
of one unit and in addition, it can transport $B$ cans of $C$ units each. 
The depots must be made of cans and only can fuel may be used to fill them.

We consider the problem of crossing a desert of $d$ miles $n = n_1 + n_2$ 
times to reach an oasis on position $d$, of which exactly $n_2$ times in 
both ways. There are $n + m$ Dewdney jeeps to do so, of which exactly 
$n_1 + m_1$ jeeps do not need to return to the desert border finally.
Instead, they may end up anywhere in the desert.
 
In the backward convoy algorithm for normal jeeps, we did not 
care in what jeep fuel was transported, since there were no restrictions
on transferring fuel from one jeep to another. In this sense, fuel
was global, i.e.\@ belonged to the whole convoy. In the backward convoy
formulation with Dewdney jeeps, we will only see the can fuel
as global fuel. But fuel in a jeep's fuel tank is local to that jeep.

We allow jeeps to consume can fuel directly, instead of by way of
their fuel tank. This is no problem, since instead the jeeps can get
a little fuel from the cans such that the convoy can advance a little
farther. Despite the fact that tank fuel is local, we will allow jeeps to 
consume tank fuel from other jeeps in a {\em backward convoy algorithm for 
Dewdney jeeps}, but only if the receiving jeep is created at the lowest 
position. 
The idea behind this is the following: the receiving jeep is actually created 
on an even later moment in the backward convoy and its transportation before its 
real creation is done by means of to and fro's of the jeep which tank fuel is 
used. 

Furthermore, we allow double jeeps to be replaced by single jeeps in the
backward convoy algorithm for Dewdney jeeps. If such a replacement takes place
on position $p$, then the following happens in a normal algorithm:
the jeep rides from position $0$ to a position farther than $p$ first and
returns to position $p$ finally.

If on a position, both a single and a double jeep are created, then the double
jeep may consume fuel from the tank of the single jeep.
So the fuel of the backward convoy can be ordered by applicability. 
The tank fuel of the jeep that is created the last in the backward convoy 
is the least applicable, and the can fuel as well as the tank fuel of the $n_1$
initial single jeeps is the most applicable, provided $n_1 > 0$. 

If we include fuel depots with fuel to be used in the backward convoy
algorithm for Dewdney jeeps, then cans may become scarce, just as with 
Maddex' jeep problem. A solution is to use jeeps that have a tank 
capacity of one unit and can transport $BC$ units of fuel in addition, 
fuel that may be put in any proportion on the desert just as the fuel of 
Fine jeeps. Next, we could try to formulate an algorithm for one jeep
just as in section \ref{Finesec}.

But besides saturation points, there is other trouble that can
occur. Suppose we have a depot with fuel to be used. Now it seems optimal to
use that fuel to cancel as many double jeeps as possible. But in order to 
do that, these double jeeps should be able to use all tank fuel available.
This is impossible, since tank fuel is not global. For this reason, we do not
include depots with fuel to be used in the backward convoy algorithm for 
Dewdney jeeps. 

\begin{algo} \label{Dewdneyjeep}
Start with a convoy at position $d$ with $n_1$ single jeeps and 
$n_2$ double jeeps initially, all jeeps with $B$ full cans of
$C$ units of fuel each and one unit of tank fuel.
If a depot has to be filled at position $d$, then 
call the handler of event 1 first. After that, ride to position $0$ 
with the whole convoy. Each jeeps consumes the least applicable fuel it can
consume, so a jeep use its own tank fuel first, then other tank fuel and at 
last can fuel.

\begin{em}
Event 1: The convoy meets a position where a depot has to be filled.
\end{em} \\
{\em Handler:} Use can fuel to make the depot. If there is not
enough fuel to make the depot, then call the handler of event 2 first
and then repeat this handler. Otherwise, advance to $0$.

\begin{em}
Event 2: There is at least one jeep that can not ride farther any more, due
to lack of fuel.
\end{em} \\
{\em Handler:} Create a new double jeep with one unit of tank
fuel and $B$ full cans of $C$ units of fuel each.
If the convoy has less than $m_1 + n_1$ single jeeps now, then replace 
the oldest double jeep in the backward convoy by a single jeep.

Eventually at the desert border, the amount of fuel of all jeeps is made 
equal to one unit. So the result of the algorithm is the number of jeeps
minus the final amount of fuel.
\end{algo}

\section{OPTIMALITY RESULTS FOR CONVOYS OF DEWDNEY JEEPS}

In the backward convoy algorithm for Fine jeeps in \cite{opjc}, 
it was possible to remove double jeeps in an extended convoy algorithm 
if they contained exactly one unit
of fuel in their tank. Since the amount of tank fuel can not be
controlled with Dewdney jeeps, other methods are needed to remove
double jeeps, which we call {\em jeep merges}. We distinguish
two jeep merges:
\begin{enumerate}

\item A single jeep and a double jeep with $r$ units of tank fuel
      together merge to a single jeep with $r - 1$ units of fuel, 
      where $1 \le r \le 2$.
      
\item Two double jeeps with $r$ units of tank fuel together
      merge to a double jeep with $r - 1$ units of fuel, where
      $1 \le r \le 3$.
      
\end{enumerate}
For each of both merges, $B$ full cans of fuel has to be payed.
We allow these jeep merges in an {\em extended backward convoy algorithm for 
Dewdney jeeps}, but we do not allow jeeps to consume tank fuel of other jeeps
now. Although cans may be transported in backward direction in a normal 
algorithm, we do not need double jeeps to carry such things as anti-cans,
since each such can must be transported in forward direction first. Thus
with time eliminated, we can demand that the number of cans each jeep 
carries lies between $0$ and $B$ inclusive in an extended backward convoy 
algorithm.

\begin{figure}[ht] \label{merges}

\begin{center}
\begin{picture}(300,340)(0,0)

\put(63,-5){\epsfbox{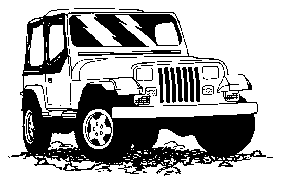}}
\put(198,55){\epsfbox{jeep.ps}}
\put(138,115){\epsfbox{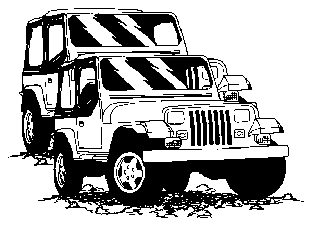}}
\put(220,185){\epsfbox{jeep.ps}}
\put(10,245){\epsfbox{jeeps.ps}}
\put(95,275){\epsfbox{jeeps.ps}}
\put(185,275){\epsfbox{jeeps.ps}}

\put(0,0){\makebox(0,0){$\bullet$}}
\put(1.5,6){\vector(1,4){57}}
\put(60,240){\makebox(0,0){$\bullet$}}
\put(60,60){\dashbox(0,180)}
\put(61.5,234){\vector(1,-4){42}}
\put(60,60){\dashbox(180,0)}
\put(105,60){\makebox(0,0){$\bullet$}}
\put(105,40){\dashbox(0,20)}
\put(106.5,66){\vector(1,4){49.5}}
\put(157.5,270){\makebox(0,0){$\bullet$}}
\put(157.5,180){\dashbox(0,90)}
\put(159,264){\vector(1,-4){19.5}}
\put(180,180){\makebox(0,0){$\bullet$}}
\put(180,172){\dashbox(0,8)}
\put(157.5,180){\dashbox(45,0)}
\put(181.5,186){\vector(1,4){19.5}}
\put(202.5,270){\makebox(0,0){$\bullet$}}
\put(202.5,180){\dashbox(0,90)}
\put(204,264){\vector(1,-4){34.5}}
\put(240,120){\makebox(0,0){$\bullet$}}
\put(240,100){\dashbox(0,20)}
\put(180,120){\dashbox(75,0)}
\put(241.5,126){\vector(1,4){12}}
\put(255,180){\makebox(0,0){$\bullet$}}
\put(255,120){\dashbox(0,60)}

\end{picture}
\end{center}

\caption{An illustration of several types of jeep merges}
\end{figure}

Suppose that some jeep turns from backward to forward at some position
$p > 0$. If that jeep passes $p$ after this turn, then this turn corresponds
to merge 2, otherwise it corresponds to merge 1. Some jeep merges are
illustrated in figure \ref{merges}. A jeep that rides back to
$p > 0$ and stays there corresponds to a creation of a single jeep at $p$,
immediately followed by a merge of type 1 at position $p$. 

\begin{proposition}
A normal algorithm for Dewdney jeeps can be transformed to an
extended backward convoy algorithm for Dewdney jeeps.
\end{proposition}

\begin{proof}
The proof is similar to that for Fine jeeps instead of Dewdney jeeps 
in \cite[Prop.\@ 3.1]{opjc}, and therefore omitted.
\end{proof}

The problem of making a normal algorithm for Dewdney jeeps from a 
backward convoy algorithms for Dewdney jeeps will be discussed in 
section \ref{lessDewdney}. For that reason, the word `optimal'
in the theorems below should be read as `at least as good as an
optimal normal algorithm' for the present.

\begin{theorem} \label{Dewdneyconvoy}
Algorithm \ref{Dewdneyjeep} is optimal.
\end{theorem}

\begin{proof}
Consider an extended backward convoy algorithm for Dewdney jeeps.
We do not transfer can 
fuel to fuel tanks until position $0$: jeep must consume can fuel directly 
instead. So all jeeps have at most one unit of tank fuel each. 
Furthermore, we impose the following harmless assumption on the backward convoy 
algorithm: we postpone adding a jeep to the convoy
until no can fuel is left, except if the added jeep
is merged immediately.

We first show that jeep merges always coincide with jeep 
creations, i.e.\@ each jeep merge implies a jeep creation at the same
place.
Suppose that this is not the case. Say that at the first 
merge that does not coincide with a jeep creation,
some jeep $A$ with $x$ units
of tank fuel is merged with a jeep $B$ with $y$ units of tank 
fuel, where $x \le y$. 

If the number of jeeps of the backward 
convoy is $n$ before the merge, then the merge will make the 
number of jeeps of the convoy too small; contradiction.
If the number of jeeps of the convoy is more
than $n$, then at least one jeep has already been added to the 
initial convoy for another reason than a merge.
The only reason for that jeep addition can be that the convoy
ran out of can fuel. It follows that the amount of can 
fuel is less than $B \cdot C$ now, which is not enough for a
merge; contradiction.

A merge implies a jeep creation at the same place, but both do not cancel out 
in general. Suppose we have a creation of a jeep and a merge of two other
jeeps. The amount of tank fuel of the 
created jeep is equal to one, but the merging jeeps might have other
amounts of tank fuel. Say that at the first merge of the backward convoy,
a jeep $A$ with $x$ units of fuel merges
with a jeep $B$ with $y$ units of tank fuel, where jeep $A$ is created
before jeep $B$. The jeep merge results in a jeep with 
$x + y - 1$ units of tank fuel. Simultaneously, a new jeep with one 
unit of tank fuel must be created. 
Since we do not transfer can fuel to fuel tanks before position 0,
$y \le 1$.
Therefore, the above jeep merge and jeep creation can be simulated
by transferring $1 - y$ units of tank fuel from jeep $A$ to jeep $B$,
if we allow jeeps to change type between single and double jeep.

Subsequent jeep merges can be simulated in a similar matter.
So if we allow that tank fuel is transferred from a jeep $A$ to a jeep $B$
if $A$ is created before $B$, as well as changes of type, then jeep 
merges are no longer necessary.
Instead of transferring tank fuel from jeep $A$ to jeep $B$, we allow jeep
$B$ to consume the tank fuel of jeep $A$ directly, just as in algorithm
\ref{Dewdneyjeep}. 

Since tank fuel may be transferred from older jeeps to 
newer jeeps and single jeeps are more economical than double jeeps, we may 
assume that the oldest jeeps are single jeeps and the 
newest jeeps are double jeeps. This is the case in algorithm \ref{Dewdneyjeep}.
Furthermore, the total number of jeeps is minimal and as many jeeps as possible
are single jeeps all the time in algorithm \ref{Dewdneyjeep}. So algorithm 
\ref{Dewdneyjeep} is optimal.
\end{proof}

\section{A NORMAL ALGORITHM FOR DEWDNEY JEEPS} \label{lessDewdney}

Let us first turn algorithm \ref{Dewdneyjeep} into a normal convoy algorithm under
the assumption that the number of jeeps is unlimited.
The problem that must be overcome is that transferring tank fuel from one jeep
to another is not allowed. Fortunately, the jeep that receives fuel is always
a double jeep and the jeep from which tank fuel is taken is always created earlier 
in algorithm \ref{Dewdneyjeep}. Therefore, the farthest miles of the receiving 
double jeep can be done by to and fro's of the jeeps corresponding to the 
single and double jeeps from which tank fuel is taken. 

However, some of the $n_2$ initial double jeeps might consume tank fuel from
other jeeps, which makes that the $n_2$ round trips might be scattered. This is 
solved by taking all to and fro's together, and next apply the split lemma to 
give each jeep that perform to and for's an interval that corresponds to 
its moment of creation in the backward convoy algorithm. It is tank fuel that is
used for these to and fro's, thus the jeeps doing them do not need to be with each
other to share a can.

Notice that also jeeps
corresponding to double jeeps can be ordered to do to and fro's: the double 
jeep becomes a single jeep later in the backward convoy, saving fuel
with respect to the double jeep that receives the tank fuel. 
As said before, a double jeep that is replaced 
by a single jeep at position $p > 0$ in algorithm \ref{Dewdneyjeep} 
corresponds to a jeep that reaches farther than position $p$ and finally 
ends at position $p$ in a normal algorithm.

So we can replace algorithm \ref{Dewdneyjeep} by a backward convoy algorithm 
where no transfers of tank fuel occur, if we see single and double jeeps 
from which tankfuel is used as triple and quadruple jeeps respectively or 
bigger. Subsequently, we can formulate a normal algorithm for an unlimited number
of jeeps, where those jeeps ride together as far as they are riding on that moment.

Next, assume that there is a limited number of jeeps. Then
double jeeps that are replaced by single jeeps might be a serious problem.
For that purpose, we start with assuming that $m_1 = 0$. The next problem
is that the round trip might be scattered. Therefore, we additionally
assume that there are at least $n$ jeeps.

Thus cosider the case that there are exactly $n$ jeeps. 
Notice that jeeps do not use tank fuel of 
other jeeps any more after the creation of the $(n+1)^{\rm th}$ jeep in the backward 
convoy. Think of the $(n+1)^{\rm th}$ jeep as a round trip before
the convoy rode out, possibly the farthest reaching jeep of another convoy
of at most $n$ jeeps. The $(n+1)^{\rm th}$ jeep is needed for fuel, so a can is 
opened. 

Now the jeeps in the backward convoy can absorb their portions of
the fuel of this can as long as it does not exceed their tanks. If that is
the case, the can does not need to be dragged with the backward convoy and 
it is thus no problem that the convoy is split in sub\-convoys of 
at most $n$ jeeps in a normal algorithm, provided we assume that 
the $(n+1)^{\rm th}$ jeep fills the depots with their portions of fuel from that 
same can in its outward trip.

More generally, making a normal algorithm is straightforward,
as long as the jeeps of the backward convoy can absorb their portions of
fuel of the cans after the creation of the $(n+1)^{\rm th}$ jeep all the time.
So assume that at some point in the backward convoy, a can is opened
with more can fuel than the jeeps can absorb. Say that $x$ is the
position closest to the desert border where this happens. Let $y$
be the position where the next can is opened in the backward convoy,
in case that occurs, and $y = 0$ otherwise.

We can get a normal algorithm as follows. First, the jeeps ride out in 
convoys of at most $n$ jeeps to positions less than $y$, returning to
the desert border, provided $y > 0$. Next, all $n$ jeeps ride to $y$ 
and next to $x$, where some jeeps ride to and fro's between $y$ and $x$, but
no jeep gets too far away from the jeep with the can that is opened at position $x$
in the backward convoy. After that, all riding on positions farther than 
$x$ is done, and induction tells us how, where the desert border is replaced by $x$. 
Finally, zero or more jeeps return to the desert border.

\section{DEWDNEY JEEPS WITH SEALED CANS}

We now consider the problem of Dewdney jeeps with {\em sealed cans}.
In order to take fuel from a sealed can, it must be unsealed.
But after unsealing, a can must not be transported any more. 
It is however not necessary for an unsealed can to be emptied
immediately. 

We first give a backward convoy algorithm for Dewdney jeeps with sealed cans.
Since unsealed cans can be seen as fuel depots to be used, we allow both
additional fuel depots to be used and fuel depots to be filled. 
Depot fuel to be used may be put in the fuel tanks of the jeeps. Can fuel
has to be used to fill depots with the indicated amount of fuel. 

The problem has some similarities with that of the Fine jeep in section 
\ref{Finesec}. For that problem, it was relatively hard for the jeep to 
reach a position with lots of fuel and relatively easy after reaching
this position. Since each jeep in the convoy must reach such a position,
the actual number of jeeps matters. For that reason, we assume that there
are exactly $n$ jeeps, all of which have to reach an oasis,
of which $n_1$ jeeps stay at the oasis and $n_2$ 
jeeps return to the desert border. Double jeeps that are added later
to the backward convoy are to and fro's of the above $n$ jeeps.

For the to and fro's that are global in nature, we add {\em red
double jeeps} for forward loops of the above $n$ jeeps, and
{\em blue double jeeps} for backward loops of the above $n$ jeep.
The red and blue double jeeps are replaced by each other when
the loop ends. Blue double jeeps have some similarities with the 
backward refueling sub\-convoy in section \ref{Finesec}. 

For very local to and fro's,
we introduce another type of jeep, namely {\em spider jeeps}. Spider jeeps are 
double jeeps without a fuel tank, which can consume tank fuel of other
jeeps. From above, a spider jeep that uses tank fuel of eight surrounding 
regular jeeps looks like a spider. Spider jeeps may replace or may be replaced by
red and blue double jeeps.

\begin{algo} \label{Dewdneysealed}
Start with a convoy at position $d$, with $n_1$ single
jeeps with one unit of tank fuel and $n_2$ regular double jeeps with one unit
of tank fuel initially. All jeeps get $B$ sealed cans of $C$ units of 
fuel each. If there is a depot at $d$ to be used, then call the handler of 
event 1 first. Ride to position $0$ with the whole convoy.

Spider jeeps that are created along the way consume fuel 
of any other type of jeep with relatively the most fuel, until all such jeeps
have at most 50 percent of tank fuel. At that point, all spider jeeps are
replaced by red double jeeps.

Blue jeeps that get only one unit of tank fuel are removed from the backward
convoy in case there are $B$ sealed cans of $C$ units available to do so. 
Otherwise, they are replaced by spider jeeps if there is a jeep with more than 
50 percent of tank fuel at that moment and by red jeeps if all jeeps have 
at most 50 percent of tank fuel.

Due to riding, the relative amount of tank fuel decreases at a rate that
is equal for all jeeps. Since spider jeeps affect the relative amount of tank fuel
of the relatively fullest jeep, the difference in relative amount of tank fuel
cannot get larger than 50 percent between two jeeps. Neither do the events below 
affect this difference property. 

\begin{em} 
Event 1: The convoy meets a depot to be used.
\end{em} \\ 
{\em Handler:} 
Distribute the fuel amongst the tanks of regular jeeps and red double jeeps 
in the backward convoy, such that the minimum relative amount of tank fuel becomes as large as 
possible, but do not give fuel when tanks get or are more filled than 50 percent. 
If the relative amount of tank fuel 
for these jeeps gets 50 percent, then replace all red double jeeps by 
spider jeeps.

If there is still fuel left, then 
distribute the fuel amongst all regular jeeps, such that the minimum relative 
amount of tank fuel of the regular jeeps becomes as large as possible.
If there is more fuel than these jeeps can accept, then fill the tank
of the blue double jeep with the most tank fuel.
If there is still fuel left, then advance with the blue double jeep with
the next most tank fuel, etc.

If there is still fuel left,
then replace a spider jeep by a blue double jeep with one unit of tank fuel
and fill that blue jeep up to two units. If this is not enough to absorb
all fuel, then do the same with another spider jeep, etc. If there is still fuel 
left, then create a new blue double jeep with one unit of tank fuel and fill its
tank. Create blue double jeeps until all fuel can be absorbed. 
Advance to $0$.

\begin{em}
Event 2: There is at least one jeep that gets out of fuel and can 
not ride farther any more.
\end{em} \\ 
{\em Handler:}
Notice that there are no jeeps that are more than half filled, and thus no blue 
jeeps or spider jeeps. Unseal a can and do the handler of event 1, seeing the 
opened can as a fuel depot. If there is no can to be unsealed, then create
a new red double jeep with one unit of tank fuel and $B$ sealed cans of $C$
units of fuel each first. Advance to $0$.

\begin{em}
Event 3: The convoy meets a depot to be filled.
\end{em} \\ 
{\em Handler:}
Unseal as many cans as necessary to fill the depot. If there is still 
fuel needed, then create as many spider jeeps with $B$ cans of $C$ units 
as necessary. In the last can, some fuel might remain.

Assume first that there is no jeep with more than 50 percent of tank fuel.
Then replace all (above) spider jeeps by red double jeeps with one unit of tank fuel.
After that, perform the handler of event 1 for the remainder of the last can, 
if there is. Next, advance to 0.

Assume next that some jeep has more than 50 percent of tank fuel.
If some regular double jeep has more than 50 percent of tank fuel and 
a spider jeep was created to provide some cans, then replace
the spider jeep that provided the last can by a blue double jeep
with one unit of tank fuel, and use the remainder of the last can to
fill the tank of this blue double jeep, until the last can is empty or
the blue double jeep has the same amount of fuel as the regular double 
jeep.

Next, if some can fuel remains in the last can, then perform the handler 
of event 1 for it. After that, advance to 0.

Eventually at the desert border, the amount of fuel of all jeeps is made 
equal to one unit. So the result of the algorithm is the number of jeeps
minus the final amount of fuel.
\end{algo}

In algorithm \ref{Dewdneysealed}, the types of jeeps indicate more or less how
a normal algorithm for Dewdney jeeps with sealed cans should look. Furthermore,
the cans do not move any more after being opened, so there is not much difference 
between them in algorithm \ref{Dewdneysealed} and in a normal algorithm.
Nevertheless, making a normal algorithm of algorithm \ref{Dewdneysealed}
is not always possible. 

One problem occurs in event 3, in case some fuel
remains after filling the depot and there are blue double jeeps.
In that case, the can of the remaining fuel must be from a regular jeep
or red jeep, since otherwise the regular jeeps and red jeeps are too early 
to take advantage of the fuel that remains after filling the depot.

Thus is it a good idea to use the cans of blue double jeeps first,
but not those of the blue double jeeps that will be canceled, since
all the $B$ cans of such a jeep are needed for that. Another option is that
cans are transferred from a blue double jeep or a spider jeep to a regular
jeep, but this is not always possible. It is however possible for
blue jeeps at creation and for spider jeeps when there are no blue jeeps.

Talking about the spider jeeps, there is another problem that might occur
when making a normal algorithm of algorithm \ref{Dewdneysealed}. The problem
is that a spider jeep must use tank fuel of blue jeeps before tank fuel
of regular jeeps, since the blue jeeps ride later in a normal algorithm.
Sometimes, even the order of blue jeeps matters. In section \ref{sealednormal}, 
we will however show that this problem can be overcome.

Another problem occurs when in the last mile of algorithm \ref{Dewdneysealed}, 
the miles closest to the desert border,
the amount of tank fuel in one of the double jeeps becomes
larger than $x + 1$ on position $x$. We discuss this problem in the rest of 
this section.

Although the backward convoy algorithm can not always be seen as
a normal algorithm (i.e.\@ more fuel is required for a normal algorithm),
we roughly describe a way to make a normal algorithm
from a backward convoy algorithm first. We start with a convoy of $n_1 + n_2$
jeeps riding the outward part and $n_2$ jeeps riding the return part.
But then, we do not have enough transportation. 

Therefore, we add forward
loops and backward loops to make the normal algorithm complete. 
A blue double jeep implies a backward loop there and a red double
jeep implies a forward loop there.

We now show that algorithm \ref{Dewdneysealed} can not always be seen 
as a normal algorithm. Let $n_1 = 0$ and suppose that there 
are $n_2 = 5$ double jeeps when the backward convoy reaches position
$1$. Suppose that all $5$ jeeps are empty then and there are only $2$ 
cans of $1$ unit of fuel each. Then both cans are unsealed and the 
convoy can ride to $\frac45$ before getting empty again. At position 
$\frac45$, a new double jeep is created and a new can is unsealed. The 
convoy rides farther to position $\frac34$. Next, assume that there is 
an unlimited depot at position $\frac34$. All jeeps can get completely
filled at position $\frac34$ and can ride to position $0$, which they
reach with half a unit of tank fuel each. 

In order to be canceled,
another half a unit of fuel is needed for each jeep, and of course,
the cans that are used in the algorithm. But in a normal algorithm,
the $n_2 = 5$ initial jeeps need more fuel: at least $\frac34$ units of
fuel each to be able to reach the unlimited depot at $\frac34$ from
position $0$. The bound on the amount of fuel is only met if 
returning tank fuel is allowed at position $0$. So it seems a 
good idea, not to allow more than $1 + x$ units of tank fuel 
in initial double jeeps, where $x$ is the position
of the jeeps. Similarly, it seems a good idea to bound the amount of
fuel in an added double jeep by $1 + 2x$, since otherwise such a jeep
reaches position $0$ with more than $1$ unit of tank fuel and can
not be canceled. But on $\frac34$, the amount
of tank fuel of the jeep that is added on position $\frac45$
is bounded by $2$ rather than $\frac52$.

This way, $5 \cdot \frac34 + \frac12 = 4 \frac14$ units of 
fuel are needed at position $0$. But in the following way, the $\frac12$
unit of fuel of the added double jeep can be saved and only $3\frac34$ 
units are needed at position $0$, i.e.\@ $4$ instead of $5$ cans. First, the 
$5$ jeeps ride to position $\frac34$, using $\frac34$ units of fuel
from position $0$ each. Then, they ride to position $1$ and completely
refill there, by way of $\frac54$ units of can fuel. For that purpose,
$2$ cans of $1$ unit of fuel are unsealed at position $1$, so $\frac34$ 
units of can fuel remains. After that, the jeeps go to the oasis and ride 
back to position $1$. Until now, the difference with algorithm 
\ref{Dewdneysealed} is that the can at position $\frac45$ is not
transported yet and can fuel from position $1$ is used instead.

The $5$ jeeps reach position $1$ empty in the return trip. Now,
only $3$ of them ride farther, using the remaining $\frac34$ units
of can fuel at position $1$. These three jeeps just reach the unlimited
depot at position $\frac34$.
From position $\frac34$, the first jeep rides to position $0$, carries a can
from $0$ to $\frac18$ and rides back to $0$. The second jeep rides to 
$\frac18$, carries the can there to $\frac14$ and rides to $0$.
The third jeep fetches the can at position $\frac14$
and rides to position $1$. All
remaining $3$ jeeps ride back to position $0$ using fuel from the last 
can and the unlimited depot.

So we get the following question. Does the above method of restricting
the amount of tank fuel from double jeeps work in case there are no real depots
in the last mile of the backward convoy algorithm, but only unsealed cans?
The answer is affirmative, which the reader may show.

\section{OPTIMALITY RESULTS FOR DEWDNEY JEEPS WITH SEALED CANS}

We first show that jeep merges do not work with depots of fuel to be used.
Let $x$ be a position with a depot of fuel to be used. Say that $2^k$ double
jeeps get completely filled. When those jeeps reach $x - \frac12$, they have
become half filled, and they can merge to $2^{k-1}$ completely filled jeeps.
When these $2^{k-1}$ double jeeps reach $x-1$, they can merge to $2^{k-2}$ 
completely filled jeeps, etc.

For that reason, we do not allow jeep merges in an {\em extended backward 
convoy algorithm for dewdney jeeps with sealed cans}. Instead, we allow a 
double jeep with one unit of tank fuel and $B$ cans with $C$ units of 
fuel to be canceled. In addition, we allow spider jeeps to consume tank fuel 
of other jeeps.

\begin{lemma} \label{graphlem}
Assume $f$ is a nowhere constant function of time with finitely many local 
extrema, of which one local minimum in the interior
of its time interval. If $f$ has a global minimum at the beginning of
its interval and another global extremum at the end, then up to a 
transformation of time, the graph of $f$ has one of the following subgraphs.

\begin{center}
\begin{picture}(80,80)
\put(0,10){\line(1,0){80}}
\put(10,40){\line(1,2){20}}
\put(10,40){\circle*{3}}
\put(30,80){\line(1,-2){20}}
\put(30,80){\circle*{3}}
\put(50,40){\line(1,2){20}}
\put(50,40){\circle*{3}}
\put(70,80){\circle*{3}}
\put(10,10){\dashbox(40,30){}}
\put(30,10){\dashbox(40,70){}}
\put(10,8){\makebox(0,0)[tc]{$t_4$}}
\put(30,8){\makebox(0,0)[tc]{$t_1$}}
\put(50,8){\makebox(0,0)[tc]{$t_2$}}
\put(70,8){\makebox(0,0)[tc]{$t_3$}}
\end{picture}
\quad 
\begin{picture}(10,80)
\put(5,40){\makebox(0,0){or}}
\end{picture}
\quad
\begin{picture}(80,80)
\put(0,10){\line(1,0){80}}
\put(10,80){\line(1,-2){20}}
\put(10,80){\circle*{3}}
\put(30,40){\line(1,2){20}}
\put(30,40){\circle*{3}}
\put(50,80){\line(1,-2){20}}
\put(50,80){\circle*{3}}
\put(70,40){\circle*{3}}
\put(10,10){\dashbox(40,70){}}
\put(30,10){\dashbox(40,30){}}
\put(10,8){\makebox(0,0)[tc]{$t_1$}}
\put(30,8){\makebox(0,0)[tc]{$t_2$}}
\put(50,8){\makebox(0,0)[tc]{$t_3$}}
\put(70,8){\makebox(0,0)[tc]{$t_4$}}
\end{picture}
\end{center}

\end{lemma}

\begin{proof}
Starting at the beginning of the time interval, local maxima and minima vary.
These extrema cannot get closer and closer to each other, so we have the following:
\begin{center}
\begin{picture}(100,80)
\put(0,0){\line(1,4){20}}
\put(20,80){\line(1,-3){20}}
\put(40,20){\line(2,5){20}}
\put(60,70){\line(1,-2){20}}
\put(80,30){\line(2,5){20}}
\put(44,30){\dashbox(52,40){}}
\end{picture}
\quad 
\begin{picture}(10,80)
\put(5,40){\makebox(0,0){or}}
\end{picture}
\quad
\begin{picture}(80,80)
\put(0,0){\line(1,4){20}}
\put(20,80){\line(1,-3){20}}
\put(40,20){\line(2,5){20}}
\put(60,70){\line(1,-3){20}}
\put(23.3333,20){\dashbox(53.3334,50){}}
\end{picture}
\end{center}
\end{proof}

\begin{proposition}
A normal algorithm for Dewdney jeeps with sealed cans can be 
transformed to an extended convoy algorithm for Dewdney jeeps 
with sealed cans.
\end{proposition}

\begin{proof}
Notice first that each jeep can do the things it should do by 
means of finitely many changes of direction. So there
are only finitely many local extrema. Furthermore, a normal 
algorithm for which the jeeps do not have interior local minima 
is already an extended convoy algorithm for Dewdney jeeps 
with sealed cans. So assume that one of the jeeps, say jeep $J$, 
has an interior local minimum. Then the graph of that jeep has a 
subgraph as in lemma \ref{graphlem}, where $t_1 < t_2 < t_3$ and 
$t_4 \notin [t_1, t_3]$. We distinguish three cases:
\begin{itemize}

\item[i)] There is a position in the subgraph where jeep $J$ gets
twice with the same amount of tank fuel. \\
Then we can cut off the part in between the two moments at that 
position from the path of $J$ and replace it by a double jeep.
Since the part that is cut off must contain a refuel position (possibly 
on the edge), the number of times jeep $J$ gets on a refuel position 
decreases.

\item[ii)] Jeep $J$ has more tank fuel on moment $t_3$ than on moment 
$t_1$. \\
Assuming that we do not have case i), we see that jeep $J$ has more fuel
directly after $t_2$ than directly before $t_2$. Thus jeep $J$ refuels
its tank on moment $t_2$. Now move as much as possible of this refueling
to moment $t_4$. This results in a completely filled jeep $J$ before $t_2$,
a completely empty jeep $J$ after $t_2$, or no refueling any more on moment
$t_2$. 

If ii) is not preserved during the process, then restore some refueling 
on moment $t_2$ to make that jeep $J$ has the same amount tank fuel on 
moment $t_3$ as on moment $t_1$, such that case i) applies. So assume that 
ii) is still satisfied. The one can verify that i) applies in all three
cases after transferring tank refueling from $t_2$ to $t_4$.

\item[iii)] Jeep $J$ has less tank fuel on moment $t_3$ than on moment 
$t_1$. \\
Let $p_1$ be the position where the jeep is on $t_2$ and $t_4$, and
$p_2$ that on $t_1$ and $t_3$.
Let $p_3$ be the smallest refuel position of jeep $J$ between
$t_1$ and $t_3$ if there is such a position, and take $p_3 = p_2$
otherwise. Assuming that we do not have case i), we see that jeep $J$ has less fuel
directly before $t_2$ than directly after $t_2$. Thus jeep $J$ does
not refuel on position $p_1$ on moment $t_2$, i.e.\@ $p_3 > p_1$.

Now cut off as much as possible
from the round trip from $p_3$ to $p_1$ by jeep $J$ as possible, and 
replace it by a spider jeep that uses tank fuel of jeep $J$, in the 
neighborhood of $t_4$ where jeep $J$ rides from $p_1$ to $p_3$ 
(in case $t_4 < t_2$) or from $p_3$ to $p_1$ (in case $t_4 > t_2$). 

This results in a completely empty jeep $J$ before $t_2$,
a completely filled jeep $J$ after $t_2$, or that jeep $J$
is on position $p_3$ on moment $t_2$. In the first two cases,
either i) or ii) applies. In the last case, jeep $J$ gets
one time less on refuel position $p_3$ as before if $p_3 < p_2$,
and a better path in case $p_3 = p_2$ is not a refuel position.

\end{itemize}
By repeatedly focusing on interior local minima and reducing them, we get
an algorithm where the jeeps do not have interior local minima.
Thus we finally have an extended backward convoy algorithm for Dewdney jeeps 
with sealed cans.
\end{proof}

\begin{theorem} 
Algorithm \ref{Dewdneysealed} is optimal.
\end{theorem}

\begin{proof}
As soon as a blue jeep is created in algorithm \ref{Dewdneysealed},
the backward convoy consumes the tank fuel above the level of 50 percent
first. Furthermore, given that all available fuel on the road is used,
the blue jeeps are canceled as soon as possible if there are enough cans 
to do so. This is because there are no spider jeeps when there
are $B$ or more full cans.

If there are no blue jeeps, then the opening of a can for
refueling purposes is postponed as long as possible in 
algorithm \ref{Dewdneysealed}, and the use of spider jeeps is in 
such a way that subsequent unsealings for refueling are postponed as 
long as possible, too.

Since unsealing cans is postponed as long as possible in algorithm,
\ref{Dewdneysealed}, so is the creation of a new red double jeep or spider 
jeep for getting new cans. Thus algorithm \ref{Dewdneysealed} is optimal.
\end{proof}

\section{A NORMAL ALGORITHM FOR DEWDNEY JEEPS WITH SEALED CANS} \label{sealednormal}

Let $p_0 = 0$ and $p_1, p_2, \ldots$ be the positions $> 0$ where the
regular jeep with the least relative amount of tank fuel has 50 percent 
of tank fuel in algorithm \ref{Dewdneysealed}, in increasing order. 
If the oasis to be reached is not included in the positions
$p_i$, then add that as well, to obtain a finite sequence
$$
0 = p_0 < p_1 < p_2 < \cdots < p_k
$$
If we do not count the to and fro's for spider jeeps, then the general 
scheme is the following. First, the $n$ jeeps do all riding
and transportations within $[0, p_i]$, except for $n_2$ jeeps riding
from $p_i$ to $0$. Next, the jeeps do all riding and 
transportations within $[p_i, p_{i+1}]$ except for $n_2$ jeeps riding
from $p_{i+1}$ to $p_i$. After that, the jeeps do all riding and 
transportation within $[p_{i+1}, p_k]$. Finally, $n_2$ jeeps
ride from $p_{i+1}$ back to $0$.

Notice that red double jeeps have at least as much fuel as regular double jeeps,
but no more than 50 percent. On the other hand, blue double jeeps
have at most as much fuel as regular double jeeps, but no less than 50
percent. From this, it follows that all double jeeps of any type 
have 50 percent of tank fuel on $p_i$ for each $0 < i < k$, and
that the schemes for each $i$ are compatible with each other.

So we only need to
describe the riding and transportations within $[p_i, p_{i+1}]$.
There are two cases: the relatively emptiest regular jeep has
less than 50 percent of tank fuel or this jeep has more than
50 percent of tank fuel in $]p_i, p_{i+1}[$.

Assume first that the relatively emptiest regular jeep has
less than 50 percent of tank fuel in $]p_i, p_{i+1}[$.
Then all regular jeeps with the same multiplicity (single or double)
have the same less than 50 percent of tank fuel in 
$]p_i, p_{i+1}[$. Furthermore, all double jeeps that are created
in $[p_{i+1},p_k]$ have 50 percent of tank fuel on $p_{i+1}$ and
all double jeeps that are created in $]p_i,p_{i+1}[$ are spider jeeps
or red jeeps that have 50 percent of tank fuel when they are created.

So all riding of red double jeeps between $p_i$ and $p_{i+1}$ can be done
by forward loops form $p_i$. Perform these forward loops first  and
ride with $n$ jeeps from $p_i$ to $p_{i+1}$ after that, where the
$n_1$ single jeeps perform the to and fro's of the spider jeeps along 
the road.

If $i = 0$, then the amount of tank fuel
of the jeeps performing the forward loop and regular double jeeps
is $\min\{f(x)/2,x\}$ in the return part on position $x$ and 
$\min\{1-f(x)/2,1-f(x)+x\}$ in the outward part (instead of just 
$f(x)/2$ and $1-f(x)/2$ respectively), 
with $f(x)$ being the amount of tank fuel of 
the red or regular double jeep at hand on position $x$.
This is to ensure that jeeps return empty on position $0$.

Assume next that the relatively emptiest regular jeep has more than
50 percent of tank fuel in $]p_i, p_{i+1}[$.
Then all regular jeeps have more than 50 percent of tank fuel in 
$]p_i, p_{i+1}[$. 

Looking at the backward convoy algorithm \ref{Dewdneysealed}, we can see
how much riding and transportations need to be done on each interval.
With this, we do not count blue double jeeps that are canceled later
together with their $B$ cans. 
These jeeps can be removed from the backward convoy algorithm.
This is because spider jeeps are only present in the convoy when there
are less than $B$ cans in case there are blue jeep, thus the above blue 
jeeps do not need to give tank fuel to other jeeps.

Now that the total number of double jeeps of any type is determined 
everywhere by algorithm \ref{Dewdneysealed}, we replace event 1 of it 
by the following.

\medskip
\begin{em} 
Event 1: The convoy meets a depot to be used.
\end{em} \\ 
{\em Handler:} 
Distribute the fuel amongst the regular jeeps and red double jeeps in the backward 
convoy, such that the minimum relative amount of tank fuel becomes as large as 
possible, but not larger than 50 percent. If the relative amount of tank fuel 
for these jeeps gets 50 percent, then replace all red double jeeps by 
spider jeeps.

If there is still fuel left, then replace all spider jeeps by blue double
jeeps and create additional blue double jeeps until the number of
double jeeps of any type is what it should be. Next, distribute the fuel 
amongst all regular and blue double jeeps, under the following conditions:
\begin{itemize}
\item The regular double jeeps must not get more than $1 + x$ units
of fuel, where $x$ is the current position, unless the regular single 
jeeps and blue double jeeps get 100 percent filled.

\item The minimum amount of tank fuel for all regular jeeps and blue double jeeps is 
maximized as a first priority. Next, the minimum amount of tank fuel for all regular 
single jeeps and blue double jeeps is maximized.
\end{itemize}

Notice that the spider jeeps are replaced by blue double jeeps
on position $p_{i+1}$.
If we ignore the spider jeeps for a moment, then the blue double 
jeeps can be done as backward loops from the points where
they are created or made blue.

So assume that some spider jeep occurs in $]p_i, p_{i+1}[$. The jeeps
do not need tank fuel, since they are more than 50 percent filled 
or have enough to reach 0. Thus there is a depot to be filled. 
This is done by the spider jeep or a blue double jeep that becomes the spider 
jeep later in the backward convoy in $]p_i, p_{i+1}[$. 

Some regular jeeps and blue double jeeps perform the to and fro's of
these spider jeeps, and the order does not matter, since the regular
jeeps and blue double jeeps doing so will not get any fuel as long
as there are spider jeeps and blue jeeps with a less relative amount of
fuel than the regular jeeps.

Therefore, the $n$ regular jeeps start the spidering process, of
which the $n_2$ double jeeps only half of it. Next, the blue double jeeps
follow, and at last the other halves of the $n_2$ double jeep during
their returns to the desert border.

This way, the remainder of the last can will not be available for
the jeeps. But this remainder can be used for the transportation
of some cans by a round trip of one jeep over the last part towards 
the depot, a round trip that can be performed by one of the $n_2$
double jeeps during their returns to the desert border, or any
single jeep in case $n_2 = 0$.

\section{Conclusion}

We attained optimality results for many jeep variants, using
convoy formulations as C.G. Phipps suggested. On the other
hand, many jeep variants are still open. These jeep variants
seem much harder to solve. So a lot of research can still be
done on this topic. 

I wonder what our godfather C.G. Phipps would have thought of
this article. In his article \cite[p.\@ 462]{phipps} he writes
the following.

\begin{quotation}
The number of variations upon these problems is almost endless.
One could have rendezvous points where jeeps are to assemble.
One could consider the delivery of a certain number of jeeps to 
another supply station by caravans meet halfway. Still another
variation would be to have tank-trucks accompany the jeeps. Most
of such problems can be worked by the general principles developed
here.
\end{quotation}

Maybe, C.G. Phipps already knew most of the results attained in this
article.










\end{article}
\end{document}